\documentclass[11 pt,a4paper]{article}

\usepackage{amsmath,amsthm,amsfonts,amssymb,color}

\parskip=\medskipamount
\def\hook{{\mathchoice{\vrule height 0pt depth 0.4pt width 3pt
\vrule height 5pt depth 0.4pt \kern 3pt} {\vrule height 0pt depth
0.4pt width 3pt \vrule height 5pt depth 0.4pt \kern 3pt} {\vrule
height 0pt depth 0.2pt width 1.5pt \vrule height 3pt depth 0.2pt
width 0.2pt \kern 1pt} {\vrule height 0pt depth 0.2pt width 1.5pt
\vrule height 3pt depth 0.2pt width 0.2pt \kern 1pt} }}

\def\d{\mbox{d}}

\theoremstyle{plain}
\newtheorem{thm}{Theorem}[section]
\newtheorem{lem}[thm]{Lemma}
\newtheorem{propn}[thm]{Proposition}
\newtheorem{cor}[thm]{Corollary}

\theoremstyle{definition}
\newtheorem{defn}[thm]{Definition}
\newtheorem{xmpl}{Example}

\newcommand{\hnabla}{\hat{\nabla}}
\newcommand{\Sp}{\text{Sp}} 

\begin{document}

\title{The inverse problem
in the calculus of variations: new developments}

\author{Thoan Do and Geoff Prince
\thanks{Email: {\tt dtkthoan\char64 ctu.edu.vn,
g.prince\char64 latrobe.edu.au}}
\\Department of Mathematics and Statistics, La Trobe University,\\ Victoria 3086,
Australia.}
\date{November 21, 2019}

\maketitle

\begin{quote}
{\bf Abstract.} {\small  We deal with the problem of determining the existence and uniqueness of Lagrangians for systems of $n$ second order ordinary differential equations.
A number of recent theorems are presented, using exterior differential systems theory (EDS).
In particular, we indicate how to generalise Jesse Douglas's famous solution for $n=2$.
We then examine a  new class of solutions in arbitrary dimension $n$ and give some non-trivial examples in dimension 3.}
\end{quote}

\section*{Olga Rossi and the Ostrava Seminar}
It has been a great privilege to have worked with Olga Rossi over many years. In addition to being an outstanding mathematician and academic,  Olga was a remarkably generous and warm individual.
We have both enjoyed the hospitality of the Department at the University of Ostrava under her leadership and the seminars have always been a highlight of our visits.
We thank Pasha Zusmanovich for his stewardship of the seminar series and for his invitation to make this contribution.

\section{The inverse problem in the calculus of variations}\label{Section 1}
The inverse problem in the calculus of variations involves deciding
whether
the solutions of a given system of second-order ordinary differential
equations (SODEs)
$$
\ddot x^a = F^a (t, x^b \dot x^b), \ \ a, b = 1, \dots, n
$$
are the solutions of a set of Euler-Lagrange equations
$$
\frac{\partial^2L}{\partial \dot x^a \partial \dot x^b}
\ddot x^b + \frac{\partial^2L}{\partial x^b \partial \dot x^a} \dot x^b
+ \frac{\partial^2 L}{\partial t \partial \dot x^a} =
\frac{\partial L}{\partial x^a}
$$
for some Lagrangian function $L(t, x^b, \dot x^b)$.
Clearly the Hessian matrix $\frac{\partial^2L}{\partial \dot x^a \partial \dot x^b}$ should be invertible on some domain.
The problem dates to the end of the $19^{\text{th}}$ century and it still has
deep importance for mathematics and mathematical physics (see~\cite{KP08,HS88}).

Because the Euler-Lagrange equations are not generally in normal form, the
problem is to find a so-called multiplier matrix
 $g_{ab} (t, x^c, \dot x^c)$ which is invertible on some domain and
such that
$$
g_{ab} (\ddot x^b - F^b) \equiv \frac{d}{dt} \left(\frac{\partial L}{
\partial \dot x^a}\right) - \frac{\partial L}{\partial \dot x^a}.
$$

The most commonly used set of necessary and sufficient conditions for
the existence of the $g_{ab}$ are the so--called {\it Helmholtz
conditions} due to Douglas~\cite{D41} and put in the following form by
Sarlet~\cite{S82}:
$$
g_{ab} = g_{ba}, \quad \Gamma(g_{ab}) =
g_{ac}\Gamma_b^c+g_{bc}\Gamma_a^c, \quad
g_{ac}\Phi_b^c  = g_{bc}\Phi_a^c, \quad \frac{\partial g_{ab}}{\partial\dot x^c}  =
\frac{\partial g_{ac}}{\partial\dot x^b},
$$
where
$$
\Gamma_b^a := -{\frac{1}{2}}\frac{\partial F^a}{\partial \dot x^b},
\quad \Phi_b^a := -\frac{\partial F^a}{\partial x^b} -
\Gamma_b^c\Gamma_c^a - \Gamma(\Gamma_b^a),$$
and where
$$\Gamma := \frac{\partial}{\partial t} + u^a \frac{\partial}{\partial
x^a} + F^a \frac{\partial}{\partial u^a}.$$
When a solution $g_{ab}$ exists a corresponding Lagrangian is recovered from
$ \frac{\partial^2L}{\partial \dot x^a \partial \dot x^b}= g_{ab}.$

A full review of our perspective on the  inverse problem as at 2008 and the role of exterior differential system theory (EDS) can be found in the
article by Krupkov\'{a} and Prince ~\cite{KP08} which includes reference to other approaches. A full account of the latest developments by the current authors can be found in~\cite{Do16,DP16}.

\subsection{Timeline}

There have been too many books and papers written about this inverse problem for us to list. Instead, we offer a brief time-line of milestones in the development of our particular approach.

\begin{itemize}
\item[1886] Sonin solves the inverse problem for one equation ($n=1$)~\cite{Son}
\item[1887] Helmholtz states the problem~\cite{HH01}
\item[1898] Hirsch states the problem~\cite{Hirsch02}
\item[1941] Douglas solves the inverse problem for $n=2$~\cite{D41}
\item[1982] Henneaux \& Shepley propose an algorithm for solving the general inverse problem, identify quantum mechanical  difficulties~\cite{H82, HS88}
\item[1982] Sarlet reformulates the Helmholtz conditions~\cite{S82}
\item[1984] Crampin, Prince, Thompson geometrise the problem~\cite{CPT84}
\item[1990] Morandi et al develop the geometric framework ~\cite{MFLMR90}
\item[1992] Anderson \& Thompson apply the EDS technique and solve the first arbitrary $n$ subcase~\cite{AT92}
\item[1994] Crampin et al reframe Douglas' $n=2$ analysis  in geometric terms~\cite{CSMBP94}
\item[1994] Massa and Pagani introduce their linear connection for SODEs~\cite{MaPa94}
\item[1999] Crampin, Prince, Sarlet \& Thompson solve more arbitrary $n$ cases~\cite{STP}
\item[2003] Aldridge applies EDS to Douglas $n=2$ and some arbitrary $n$~\cite{Ald03,APST06}
\item[2016] Do and Prince identify the classification structure for arbitrary $n$ and apply it to $n=3$, 75 years after Douglas~\cite{Do16,DP16}
\end{itemize}

\section{Geometric formulation and EDS}

We will provide only enough of the geometric setting of the inverse problem to make the later discussion viable; more complete
descriptions and further references can be found in \cite{APST06, JP01,KP08}.

\subsection{$\mathbf{2^{nd}}$ order o.d.e's}
\label{SODES}

\noindent Suppose that $M$ is some differentiable manifold with
generic local co-ordinates $(x^a)$.  The {\em  evolution space} is
defined as $E:= \mathbb R \times TM$, with projection onto the
first factor being denoted by $t: E \to \mathbb R$ and bundle
projection $\pi: E \to \mathbb R \times M$. $E$ has adapted
co-ordinates $(t, x^a, u^a)$ associated with $t$ and $(x^a)$.

A system of second order differential equations with local expression
$$
\ddot x^a = F^a (t, x^b, \dot x^b), \ a, b, = 1, \dots, n
$$
is associated with a smooth vector field $\Gamma$ on $E$ given in
the same co-ordinates by
$$
\Gamma := \frac{\partial}{\partial t} + u^a
\frac{\partial}{\partial x^a} + F^a \frac{\partial}{\partial u^a}.
$$
$\Gamma$ is called a {\em  second order differential equation
field} or {\sc  SODE}. It  can be thought of as the total
derivative operator associated with the differential equations.
The integral curves of $\Gamma$ are just the parametrised and
lifted solution curves of the differential equations.
When the system admits a Lagrangian as described in section 1, $\Gamma$
is called {\em the Euler-Lagrange field.}

The evolution space $E$ is equipped with the {\em vertical
endomorphism} $S$, defined locally by $S:=V_a\otimes \theta^a$
(see~\cite{CPT84} for an intrinsic characterisation). $S$
combines the {\em contact structure} and {\em vertical
sub--bundle}, $V(E)$, of $E$, $\theta^a$ being the local contact
forms $\theta^a:=\d x^a - u^a \d t$ and $V_a:=\frac{\partial}{\partial
u^a}$ forming a basis for vector fields tangent to the fibres of
$\pi:E \to \mathbb R \times M$ (the vertical sub--bundle).

It is natural to study the deformation of $S$ produced by the flow
of $\Gamma$, $\mathcal{ L}_\Gamma S$. The eigenspaces of this
$(1,1)$ tensor field produce a direct sum decomposition of each
tangent space of $E$.  It is shown in~\cite{CPT84} that $\mathcal{
L}_\Gamma S$ (acting on vectors) has eigenvalues $0, +1$ and $-1$.
The eigenspace at a point of $E$ corresponding to the eigenvalue
$0$ is spanned by $\Gamma$, while the eigenspace corresponding to
$+1$ is the {\it vertical subspace} of the tangent space.  The
remaining eigenspace (of dimension $n$) is called the {\it
horizontal subspace}. Unlike the vertical subspaces these
eigenspaces are not integrable; their failure to be so is due to
the curvature of this nonlinear connection (induced by $\Gamma$)
which itself has components
$$
\Gamma^a_b:=-\frac{1}{ 2} \frac{\partial F^b}{\partial u^a}.
$$
The most useful basis for the horizontal eigenspaces has elements
with local expression
$$
H_a := \frac{\partial}{\partial x^a} -\Gamma^b_a\frac{\partial}{\partial u^b}
$$
so that a local basis of vector fields for the direct sum
decomposition of the tangent spaces of $E$ is $\{\Gamma, H_a,
V_a\}$ with corresponding dual basis $\{ \d t, \theta^a, \psi^a\} $
where
$$
\psi^a := \d u^a - F^a \d t + \Gamma^a_b \theta^b.
$$

The components of the curvature appear in the commutators of the horizontal fields:
$$
[H_a,H_b]=R^d_{ab} V_d
$$
where
$$
R^d_{ab}:=\frac{1}{2}\left (\frac{\partial^2 F^d}{\partial
x^a\partial u^b} - \frac{\partial^2 F^d}{\partial x^b\partial u^a}
+\frac{1}{2}\left (\frac{\partial F^c}{\partial
u^a}\frac{\partial^2 F^d}{\partial u^c\partial u^b} -
\frac{\partial F^c}{\partial u^b}\frac{\partial^2 F^d}{\partial
u^c\partial u^a} \right )\right ).
$$
In our chosen basis the curvature tensor is
$$R=R^d_{ab}\theta^a\wedge \theta^b\otimes V_d$$

It will be useful to have some other commutators:
$$
[H_a, V_b]=-\frac{1}{2}(\frac{\partial^2 F^c}{\partial u^a\partial u^b})V_c=V_b(\Gamma^c_a)V_c=V_a(\Gamma^c_b)V_c=[H_b, V_a],
$$
$$
[\Gamma, H_a]=\Gamma^b_aH_b + \Phi^b_aV_b, \qquad [\Gamma,
V_a]=-H_a + \Gamma^b_aV_b,
$$
and, of course, $[V_a,V_b]=0$.

Denoting the projectors defined by the $\mathcal{L}_\Gamma S$-induced direct sum decomposition as ${P_\Gamma, P_V} $ and $P_H$, the {\em Jacobi endomorphism}, $\Phi$, is
$$
\Phi=P_V\circ\mathcal{L}_\Gamma P_H=\Phi^a_b V_a\otimes\theta^b.
$$
The normal forms  of the component matrix $\mathbf \Phi=(\Phi^a_b)$, of $\Phi$ are fundamental to the analysis of the inverse problem. While the (1,1) tensor  $\Phi$ itself clearly has no real eigenspaces, the closely related {\em Shape Map}, $A_\Gamma$ \cite{JP01}, captures the real eigenspaces of  $\Phi^a_b$:
$$A_\Gamma = -\Phi - P_H \circ\mathcal{L}_\Gamma P_V= -\Phi^a_bV_a\otimes\theta^b+H_a\otimes\psi^a$$
and
$$A_\Gamma(X)=\mu X \iff \mu^2\theta^a(X)=-\Phi^a_b\theta^b(X) \ \text{and}\ \psi^a(X)=\mu\theta^a(X).$$
In what follows we will denote by $X^{V/H}$ the vertical, respectively horizontal, copies of eigenvectors $X^a$ of $\Phi^a_b$ belonging to $\mu^2$:
$$X^V:=X^a V_a\quad X^H:=X^a H_a$$
so that $X^H+\mu X^V$ belongs to the corresponding eigenvalue $\mu$ of $A_\Gamma$. Similarly for the eigenforms $\phi^{V/H}$.

Note: In a more complete presentation mathematical framework for the inverse problem we would also introduce the Massa and Pagani connection~\cite{MaPa94}, the shape map~\cite{JP01} and the (jet bundle) calculus along the projection~\cite{CMS96}. For an extensive review see~\cite{KP08}.

\subsection{The Helmholtz conditions}

The Helmholtz conditions given in section 1 are the necessary and sufficient conditions that
a two form $g_{ab}\psi^a\wedge \theta^b$ be closed and of maximal rank on some domain.
This can be given an even more geometric framing in the following theorem from \cite{CPT84}:

\begin{thm}\label{CPT}
Given a {\sc SODE} $\Gamma$, the necessary and sufficient conditions for there
to be Lagrangian for which $\Gamma$ is the Euler--Lagrange field is that
there should exist a 2--form $\Omega$ such that
\begin{align*}
&\Omega(V_1,V_2)=0, \quad \forall\ V_1, V_2 \in V(E)\\
&\Gamma \hook \Omega = 0\\
&\d \Omega = 0\\
&\Omega \ \text{is of maximal rank}.
\end{align*}
\end{thm}

The simplest way to see how the Helmholtz conditions arise from
theorem~\ref{CPT} is to put ${\Omega := g_{ab} \psi^a \wedge\theta^b}$ and compute $\d \Omega$:

\begin{align*}
\d\Omega & = (\Gamma (g_{ab}) - g_{cb} \Gamma^c_a - g_{ac} \Gamma^c_b) \d t \wedge \psi^a \wedge \theta^b \\
&+ (H_d (g_{ab}) - g_{cb} V_a (\Gamma^c_d))\psi^a\wedge\theta^b \wedge\theta^d \\
&+ V_c (g_{ab}) \psi^c \wedge\psi^a\wedge\theta^b \\
&+ g_{ab} \psi^a\wedge\psi^b \wedge \d t \\
&+ g_{ca} \Phi^c_b\theta^a \wedge\theta^b \wedge \d t \\
&+ g_{ca} H_b (\Gamma^c_d)\theta^a \wedge\theta^b\wedge\theta^d.
\end{align*}

The four Helmholtz conditions are
\begin{alignat*}{2}
&\d\Omega (\Gamma, V_a, V_b) = 0, \qquad \qquad &&\d \Omega (\Gamma, V_a, H_b) = 0, \\
&\d\Omega (\Gamma, H_a, H_b) = 0, \qquad \qquad &&\d \Omega (H_a, V_b, V_c)  = 0.
\end{alignat*}

The remaining conditions arising from $\d\Omega = 0$, namely
\begin{displaymath}
\d \Omega (H_a, H_b, V_c) = 0 \qquad \mbox{and} \qquad \d \Omega(H_a, H_b, H_c) = 0,
\end{displaymath}
can be shown to be derivable from the first four (notice that this last condition is void in dimension 2).

\subsection{The EDS approach}
The 1991 book by Bryant, Chern et al \cite{Bryant91} is a comprehensive reference for exterior differential systems;  in the context of the inverse problem the landmark reference is the 1992 memoir by  Anderson and Thompson~\cite{AT92}.

In exterior differential systems  terms, the inverse problem is

``{\em Find all closed, maximal rank 2-forms in} $\Sigma:=Sp\{ \psi^a
\wedge \theta^b\} \subset\bigwedge^2(E)$''

There are three steps in the EDS process:
\begin{enumerate}
\item[1.] Find the largest differential ideal generated by the submodule $\Sigma$. An algebraic and iterative process.
\item[2.] Create a Pfaffian system from the closure condition on this ideal. A differential process.
\item[3.] Apply the Cartan-K\"{a}hler theorem to determine the generality of
the solution of this Pfaffian system. A somewhat intuitive process!
\end{enumerate}

So we must find all the closed, maximal rank 2-forms
on $E$ of the form
$$
g_{ab}\psi^a \wedge \theta^b,
$$
where we may as well assume that $g_{ab}$ is symmetric.
\noindent So let $\Sigma$ be the submodule of two forms
$\Sp\{\psi^a \wedge \theta^b+ \psi^b \wedge \theta^a\}$, and let $\{\Omega^k\}$ be a subset
of two forms in $\Sigma$. Initially we take $\{\Omega^k: k=1,\dots,n(n+1)/2\}$ to be some basis for $\Sigma.$
Then the  inverse problem becomes that of finding the submodule of closed, maximal
rank two forms in $\Sigma$, i.e. finding functions $r_k$ such that
$\d(r_k \Omega^k)=0$. \textit{Note that $\{\Omega^k\}$ is a working
subset of $\Sigma$ which will shrink as we progress.}

The first EDS step is to find the maximal submodule, $\Sigma'$, of
$\Sigma$ that generates a differential ideal (that is, an ideal
closed under exterior differentiation). We will find (or not) our
closed two forms in this ideal.

We use the following recursive process: starting with the submodule $\Sigma^0 := \Sigma$ and a basis $\{\Omega^k\}$,
find the submodule $\Sigma^1 \subseteq \Sigma^0$ such that $\d\Omega \in \langle\Sigma^0\rangle$ for all non-zero $\Omega \in \Sigma^1$.
That is, find the functions $r_k$ on $E$ such that $\d (r_k \Omega^k) \in \langle\Sigma^0\rangle$
and hence $r_k \d\Omega^k \in \langle\Sigma^0\rangle$. This is an algebraic problem.

Having found these $r_k$ and hence $\Sigma^1$, we check if $\Sigma^1 = \Sigma^0$ and so is already a differential ideal.
If not, we iterate the process, finding the submodule $\Sigma^2\subset \Sigma^1 \subset \Sigma^0$
and so on until at some step, a differential ideal is found or the empty set is reached.
If, at any point during this process, it is not possible to create a maximal rank two form, then the inverse problem has no solution.
That is, if $\{\Omega^1,...,\Omega^d\}$ is a basis for $\Sigma^i$, then
$\wedge^n (\sum_{k=1}^d \Omega^k)$ must be non-zero at each step.
\section{Significant results from the differential ideal step}
As in \cite{DP16} this paper again concentrates on the case where the matrix representation, $\mathbf \Phi=(\Phi^a_b)$, of $\Phi$ is diagonalisable, which corresponds to Douglas cases I, IIa or III (see \cite{CSMBP94} and \cite{STP}).
Our choice of the basis for $\mathfrak{X}(E)$ is $\{\Gamma, X_a^V, X_a^H\}$, where $X_a^V$ and $X_a^H$ are vertical and horizontal copies of eigenvectors $X_a$ of diagonalisable $\mathbf \Phi$ (belonging to eigenvalue $\lambda_a$, possibly repeated but with a distinct label $a$ per repetition).
The corresponding copied eigenforms $\phi^{aV}$ and $\phi^{aH}$, together with $dt$, form the dual basis $\{dt, \phi^{aV}, \phi^{aH}\}.$ While it's not strictly accurate we will call $X_a^{V/H}$ and $\phi^{aV/H}$ eigenvectors and eigenforms of $\Phi.$

So we start the EDS process with the module $\Sigma^0:=Sp\{\omega^{ab}\},$ where $\omega^{ab}:=\frac{1}{2}(\phi^{aV} \wedge \phi^{bH}+ \phi^{bV} \wedge \phi^{aH}),\ 1\le a\le b \le n$, and look for the (final) differential ideal generated by $\Sigma^f$.\\

Then, having found a non-degenerate, closed 2-form  $\omega=\sum_{a \leq b}r_{ab}\omega^{ab} \in \Sigma^f,$ the multiplier $g_{ab}$ is given by
\begin{equation*}\label{conver-r-to-g}
g_{cd}=r_{ab}\phi^a_c\phi^b_d,
\end{equation*}
where $\phi^a_c$ and $\phi^b_d$ are components of eigenforms $\phi^a$ and $\phi^b$ respectively.
In this section we review significant results obtained by applying the first step of exterior differential systems, namely the differential ideal step. See the paper \cite{DP16} for details.

The exterior derivatives of eigenforms $\phi^{aV}$ and $\phi^{aH}$ are:
\begin{align*}
 d\phi^{aV} = & -\tau^{a\Gamma}_b dt \wedge \phi^{bV} - \lambda_a dt \wedge \phi^{aH} + \tau_{cb}^{aH} \phi^{bV} \wedge \phi^{cH} + \tau_{cb}^{aV} \phi^{bV} \wedge \phi^{cV} \\
  \notag            & - \frac{1}{2}\phi^{aV}(R(X_b^H, X_c^H)) \phi^{bH} \wedge \phi^{cH},\\
 d\phi^{aH} = & \ dt \wedge \phi^{aV} - \tau^{a\Gamma}_{b} dt \wedge \phi^{bH} + \tau_{cb}^{aH} \phi^{bH} \wedge \phi^{cH} - \tau_{bc}^{aV} \phi^{bV} \wedge \phi^{cH},
\end{align*}

The structure functions $\tau^{a\Gamma}_b,  \tau_{cb}^{aH}$ and $\tau_{cb}^{aV}$ are defined by these expressions and the curvature tensor is that given in section~\ref{Section 1}.

\begin{propn} \label{identity thm} 
The differential ideal step finishes at $\Sigma^0$ if and only if $\mathbf{\Phi}$ is a function multiple of the identity.
\end{propn}
{\it In the remainder of this section we assume that $\Phi$ is diagonalisable with distinct eigenvalues.}
\begin{propn}\label{di-sigma-0}\cite{DP16}
Suppose that $\mathbf \Phi$ is diagonalisable with distinct eigenvalues and eigenforms $\phi^a.$ Take $\Sigma^0=Sp\{\omega^{ab}\}$ and $\omega \in \Sigma^0$. 
Then  $\omega \in \Sigma^1$ if and only if $\omega:=\sum_{d=1}^n r_d \omega^{dd}$ and the curvature satisfies
\begin{equation}\label{curvature-condition-Sigma-1}
\sum_{\text{cyclic } abc} r_{a}\phi^{aV}(R(X_b^H, X_c^H))=0,\quad \text{for all distinct $a,b,c$, (no sum on $a$).}
\end{equation}
\end{propn}
As discussed in \cite{DP16}, we introduce $\tilde \Sigma^1:=Sp\{\omega^a:=\omega^{aa}, a=1,\dots,n\}$, not necessarily satisfying \eqref{curvature-condition-Sigma-1}, so that $\Sigma^1\subseteq \tilde\Sigma^1\subset\Sigma^0$.
The results show that for the case where $\mathbf \Phi$ is diagonalisable with distinct eigenvalues, $\tilde \Sigma^1$ is the more effective option to start the differential ideal step. As we will see, this will generate an intermediate sequence of submodules of significant value.
\begin{propn}\label{di-sigma-1}
Let $\mathbf \Phi$ be diagonalisable with distinct eigenvalues. Then the necessary and sufficient conditions for $\omega=\sum_{a} r_a\phi^{aV}\wedge\phi^{aH} \in \tilde \Sigma^1$ to have its exterior derivative in the ideal $\langle\tilde\Sigma^1\rangle$ are that, for all distinct $a,b$ and $c$ (no sum),
\begin{align}
\notag&r_a\tau^{a\Gamma}_b+r_b\tau^{b\Gamma}_a=0, \\
\notag&r_a(\tau^{aV}_{bc}-\tau^{aV}_{cb})-r_b\tau^{bV}_{ca}+r_c\tau^{cV}_{ba}=0,\\
\label{di-sigma-1-condition}&r_a(\tau^{aH}_{bc}-\tau^{aH}_{cb})-r_b\tau^{bH}_{ca}+r_c\tau^{cH}_{ba}=0,\\
\notag&r_a\phi^{aV}(R(X_c^H,X_b^H))+r_b\phi^{bV}(R(X_a^H,X_c^H))+r_c\phi^{cV}(R(X_b^H,X_a^H))=0.
\end{align}
The last of these is just \eqref{curvature-condition-Sigma-1}.
\end{propn}
If these conditions are satisfied for all $r_a$ we have:
\begin{cor}\label{firststep-condition}
For diagonalisable $\mathbf \Phi$ with distinct eigenvalues, the necessary and sufficient conditions for $\tilde \Sigma^1$ to generate a differential ideal are that, for all distinct $a,b$ and $c$,
\begin{equation*}\label{Sigma-1-diff-ideal}
\tau^{a\Gamma}_b=0,\ \tau^{aV}_{bc}=0
\end{equation*}
\end{cor}

In the remaining differential ideal steps, we define $\tilde \Sigma^{i+1}:=\{\omega\in \tilde\Sigma^i: d\omega\in \langle\tilde\Sigma^{i}\rangle\}$. Thus $\tilde \Sigma^2$ is the submodule of 2-forms in $\tilde\Sigma^1$ which further satisfy the conditions in \eqref{di-sigma-1-condition} and so $\tilde \Sigma^2 \subseteq \Sigma^1 \subseteq \tilde\Sigma^1$. The relation between the sequences $\tilde \Sigma^1 \supset\tilde \Sigma^2 \supset \dots\supset \tilde \Sigma^p\supset\dots$ and $\Sigma^1\supset\Sigma^2 \supset \dots\supset \Sigma^p\supset\dots$ is as follows.

\begin{lem}\label{relation-Sigma-Sigmatilde}
 $\tilde \Sigma^1 \supseteq \Sigma^1\supseteq\tilde \Sigma^2 \supseteq \Sigma^2 \supseteq \dots\supseteq \tilde \Sigma^p\supseteq\Sigma^p\supseteq\dots.$
\end{lem}
This lemma makes clear the computational value of the $\tilde\Sigma^i.$

The following proposition indicates the sufficient condition for degenerate solutions in the distinct eigenvalue case. This will be used to exclude the cases where there are no regular solutions.

\begin{propn}\label{degenerate-condition}
Suppose a submodule $\Sigma^f $ generates a differential ideal. If any $\omega^{a}$ is missing from $\Sigma^f$ then there is no regular solution to the inverse problem.
\end{propn}

Now we identify one of the key factors in our classification for the inverse problem: {\em integrable eigen co-distributions}.
\begin{defn}
The eigen co-distribution $D_a^\bot=Sp\{\phi^{aV},\phi^{aH}\}$ of (copied) eigenforms of $\mathbf \Phi$ is said to be (Frobenius) integrable if
\begin{equation*}
d\phi^{aV}, d\phi^{aH} \equiv 0 \ (\text{mod } \phi^{aV}, \phi^{aH}), \label{comp-in-cond-1}
\end{equation*}
equivalently
\begin{equation}
 d\omega^{a}=\xi^a_a\wedge \omega^{a}\  (\text{no sum on $a$}),\ \text{i.e.}\ d\omega^{a} \equiv 0 \ (\text{mod } \omega^{a}) \label{comp-in-cond-2}.
\end{equation}
\end{defn}

Note that
\begin{equation}\label{dxi-aa}
d\xi^a_a \equiv 0 \quad (\text{mod } \phi^{aV}, \phi^{aH}).
\end{equation}
\begin{propn}\label{comp-in-cond}
The necessary and sufficient conditions for an eigen co-distribution $D_a^\bot=Sp\{\phi^{aV},\phi^{aH}\}$ of $\mathbf \Phi$ to be (Frobenius) integrable are:
\begin{equation*}
\tau^{a\Gamma}_b=0,\ \tau^{aV}_{bc}=0,\ \tau^{aH}_{bc}=0,\ \phi^{aV}(R(X_b^H,X_c^H))=0
\end{equation*}
for all $b, c \neq a.$
\end{propn}
The following important result resolves the major problem of dealing with an arbitrary number of non-integrable eigendistributions of $\Phi.$
\begin{thm}\label{DI-first-step-cond}
Let $\mathbf \Phi$ be diagonalisable with distinct eigenvalues. Suppose there are $q$ non-integrable eigen co-distributions. If the sequence $\langle \tilde \Sigma^1 \rangle, ..., \langle \tilde \Sigma^q \rangle$ does not contain a differential ideal then there is no non-degenerate solution.
\end{thm}
\begin{proof}
Suppose that the eigen co-distributions are ordered so that the first $q$ are non-integrable.
 Firstly, if $\langle \tilde \Sigma^q \rangle$ is not a differential ideal, then no earlier $\langle \tilde \Sigma^p \rangle$ can be a differential ideal. Now each of the $n-q$ integrable $\omega^b:=\phi^{bV}\wedge \phi^{bH}$ has remained in $\tilde \Sigma^q$ since $d\omega^b=\xi^b_b\wedge\omega^b$. However $\langle\tilde\Sigma^q\rangle$ is not a differential ideal so that $dim(\tilde \Sigma^q)> n-q$.
Now $dim(\tilde \Sigma^{p+1})<dim(\tilde \Sigma^p)$ for $p<q+1$ and so $dim(\tilde \Sigma^q)\leq n-(q-1)$. Thus $dim(\tilde\Sigma^q) = n-q+1$. But $\langle\tilde\Sigma^q \rangle$ is not a differential ideal by assumption and hence $dim(\tilde \Sigma^{q+1})=n-q$ and so $\omega^1, ...,\omega^q$ are missing and no solution exists.
\end{proof}
\section{A new Classification Scheme}\label{classification-scheme}

By observation from the results of the differential ideal step, in particular from proposition \ref{identity thm} and theorem \ref{DI-first-step-cond}, we suggest a more practical classification compared with that of Douglas, especially for higher dimensional problems. Our classification is based on the diagonalsability of $\mathbf \Phi$ firstly, then the number of distinct eigenvalues and integrability of eigen co-distributions of $\mathbf \Phi$ and lastly the step at which a differential ideal is obtained. 

        \begin{description}
        \item[Case A] $\mathbf{\Phi}=\lambda I_n$. This is equivalent to $\langle \Sigma^0 \rangle$ being a differential ideal (see proposition \ref{identity thm}).
        \item[Case B] $\mathbf{\Phi}$ is diagonalisable with distinct real-valued eigenvalues. Further subcases will be divided according to the integrability of the lifted two-dimensional  eigen co-distributions of $\mathbf{\Phi}$ i.e. $q$ co-distributions are non-integrable and $n-q$ are integrable. According to our theorem \ref{DI-first-step-cond},
        if up to and including $\langle \tilde \Sigma^q \rangle$ there is no differential ideal, then there is no non-degenerate multiplier. Hence, for each $q$, the subcases to be considered are that a differential ideal is generated at step $1$, step $2$,..., up to step $q$. 
        \item[Case C] $\mathbf{\Phi}$ is diagonalisable with repeated eigenvalues. Further subdivision according to integrability will be  similar to case B above.
        \item[Case D] $\mathbf{\Phi}$ is not diagonalisable. Further subdivision depends on the integrability of normal form distributions of $\mathbf \Phi$.
        \end{description}

As an example, we will provide here our suggested classification for the inverse problem in dimension 2 compared with the classification of Douglas.
Firstly, if $\mathbf \Phi$ is diagonalisable with only one eigenvalue, then $\mathbf \Phi$ is the multiple of the identity which is Douglas case I. Secondly, if $\mathbf \Phi$ is diagonalisable with two distinct eigenvalues, we divide it into three subcases (recall that as it is shown in proposition \ref{identity thm} and corollary \ref{firststep-condition} that in this case $\langle\Sigma^0\rangle$ is not a differential ideal and $\tilde \Sigma^1$ is a differential ideal if and only if $\tau^{1\Gamma}_2=0$ and $\tau^{2\Gamma}_1=0$): the first subcase is where $\mathbf \Phi$ has both integrable eigen co-distributions, that is $\tau^{a\Gamma}_b=0, \tau^{aV}_{bb}=0$ for all $a\neq b$, then this corresponds to the ``seperated"  case IIa1 of Douglas; the second subcase is where $\mathbf \Phi$ has one integrable and one non-integrable eigen co-distributions and a differential ideal is found at step 1, that is $\tau^{1\Gamma}_2=0, \tau^{2\Gamma}_1=0$ and one of the $\tau^{1V}_{22}$ and $\tau^{2V}_{11}$ is non-zero, which corresponds to Douglas case IIa2 (``semi separated"); the third subcase is where $\mathbf \Phi$ has both non-integrable eigen co-distributions which is the most difficult case. We divide this case into two further subcases depending on the step at which a differential ideal is found as follows.
\begin{itemize}
 \item[1.] A differential ideal is found at step 1, that is $\tau^{1\Gamma}_2=\tau^{2\Gamma}_1=0$ and both $\tau^{1V}_{22}$ and $\tau^{2V}_{11}$ are non-zero. This corresponds to Douglas case IIa3 (``non-separated"),
 \item[2.] A differential ideal is found at step 2. This {\em may} correspond to case III of Douglas because it is not the case that both $\tau^{1\Gamma}_2$ and $\tau^{2\Gamma}_1$ are zero which then implies $[\bar \nabla \Phi,\Phi]\neq 0$.
\end{itemize}

The remaining case is where $\mathbf \Phi$ is not diagonalisable, and this corresponds to case IIb of Douglas.

For a full classification and solutions for the inverse problem in dimension 2 we refer to chapter 5 of \cite{Do16}.
\section{Case BNII}\label{cases-in-n}
Until recently, only the two easiest cases of Douglas, case I and case IIa1, had been solved in arbitrary dimension (see \cite{SaCraMa}, \cite{CPST} and \cite{AT92}). In \cite{DP16}, we investigated in details an extension of Douglas case IIa2 in arbitrary dimension $n$, where the matrix $\mathbf \Phi$ is diagonalisable with distinct eigenvalues with exactly $n-1$ co-distributions being integrable. We also gave two examples of the case where $\mathbf \Phi$ is diagonalisable with distinct eigenvalues with two non-integrable co-distributions in dimension 3 without giving any analysis. In this section we shall provide an analysis for this case.

Case BNII is where $\mathbf \Phi$ is diagonalisable with distinct eigenvalues (label `B') and has $2$ non-integrable eigen co-distributions (label `II'), in dimension $n$ (label `N'). As we will see there are 3 subcases. Without loss of generality, we assume that the eigen co-distributions are ordered with the $2$ non-integrable eigen co-distributions are $Sp\{\phi^{1V},\phi^{1H}\}$ and $Sp\{\phi^{2V},\phi^{2H}\}$, and the other $n-2$ eigen co-distributions, $Sp\{\phi^{\alpha V},\phi^{\alpha H}: \alpha=3,...,n\}$, are integrable. According to proposition \ref{identity thm}, $\langle\Sigma^0\rangle$ is not a differential ideal and the differential ideal step of EDS starts with $\tilde \Sigma^1:=Sp\{\omega^a:=\phi^{aV}\wedge\phi^{aH}:a=1,\dots,n\}$. Furthermore, according to theorem \ref{DI-first-step-cond} the problem has no solution if up to $\tilde \Sigma^2$, the differential ideal is not found. We will discuss this in a bit more detail.


Now starting with $\tilde \Sigma^1:=Sp\{\omega^a:=\phi^{aV}\wedge\phi^{aH}: a=1,\dots,n\}$ and computing $d\omega^a$ for each $a=1,\dots, n$ we have (with no sum on $a$)

\begin{equation*}
\begin{aligned}
d\omega^a=&d(\phi^{aV}\wedge\phi^{aH})=d\phi^{aV}\wedge\phi^{aH}-\phi^{aV}\wedge d\phi^{aH}\\
         =&\xi^a_a\wedge\omega^a +\xi^a_b\wedge\omega^b\\
         &-\tau^{a\Gamma}_b dt\wedge(\phi^{bV}\wedge\phi^{aH}+\phi^{aV}\wedge\phi^{bH})\\
         &-\tau^{aH}_{cb} \phi^{cH}\wedge(\phi^{bV}\wedge\phi^{aH}+\phi^{aV}\wedge\phi^{bH})\\
         &-\tau^{aV}_{cb} \phi^{cV}\wedge(\phi^{bV}\wedge\phi^{aH}+\phi^{aV}\wedge\phi^{bH})\\
         &+\frac{1}{2}\phi^{aV}(R(X_c^H,X_b^H))\phi^{bH}\wedge\phi^{cH}\wedge\phi^{aH},
\end{aligned}
\end{equation*}
for distinct $a,b,c$ and where
\begin{equation*}\label{n-def-xi}
\begin{aligned}
\xi^a_a:=&A^{aV}_{ab}\phi^{bV}+A^{aH}_{ab}\phi^{bH},\\
\xi^a_b:=&\tau^{aV}_{bb}\phi^{aV}+\tau^{aH}_{bb}\phi^{aH}.
\end{aligned}
\end{equation*}
Recall that for all $a,b,c$, $A^{aV/H}_{bc}=\tau^{aV/H}_{bc}-2\tau^{aV/H}_{cb}$.\\

Let $\omega=r_a\omega^a \in \tilde \Sigma^1$, where $\omega^a:=\phi^{aV}\wedge\phi^{aH}, a=1,\dots,n$. According to proposition \ref{di-sigma-1}, $d\omega \in \langle\tilde \Sigma^1\rangle$ if and only if the homogeneous system of equations \eqref{di-sigma-1-condition} are satisfied by the $r_a$.

With the assumption that $Sp\{\phi^{\alpha V},\phi^{\alpha H}:\alpha=3,\dots,n\}$ are all integrable, we have
$$\tau^{\alpha\Gamma}_c=0,\ \tau^{\alpha V}_{cb}=0,\  \tau^{\alpha H}_{cb}=0, \ \phi^{\alpha V}(R(X_c^H,X_b^H))=0, $$
for distinct $\alpha,c,b$ with $\alpha=3,\dots, n.$ It follows then the system \eqref{di-sigma-1-condition} is equivalent to 
\begin{equation}\label{n-di-sigma-1-condition-3}
\begin{aligned}
&r_1\tau^{1\Gamma}_2+r_2\tau^{2\Gamma}_1=0,\\
&r_1\tau^{1\Gamma}_\alpha=0, \\
&r_2\tau^{2\Gamma}_\alpha=0, \\
&r_1(\tau^{1V}_{2\alpha}-\tau^{1V}_{\alpha 2})-r_2\tau^{2V}_{\alpha 1}=0,\\
&r_1\tau^{1V}_{2\alpha}-r_2\tau^{2V}_{1\alpha}=0,\\
&r_1(\tau^{1H}_{2\alpha}-\tau^{1H}_{\alpha 2})-r_2\tau^{2H}_{\alpha 1}=0,\\
&r_1\tau^{1H}_{2\alpha}-r_2\tau^{2H}_{1\alpha}=0,\\
&r_1\phi^{1V}(R(X_2^H,X_\alpha^H))-r_2\phi^{2V}(R(X_1^H,X_\alpha^H))=0,
\end{aligned}
\end{equation}
for all $\alpha=3,\dots,n$.

 Note that the $r_1$ and $r_2$ are unknowns in the system \eqref{n-di-sigma-1-condition-3}, and they must all be non-zero for non-degenerate solutions. Let $\mathbf A_1$ denote the matrix of coefficients of the system \eqref{n-di-sigma-1-condition-3}. 
 Now the problem can be divided into three subcases depending on the rank of $\mathbf A_1$, which is $0$, $1$ and $2$.
  Subcase 1: if $rank(\mathbf A_1)=0$, then $\tilde \Sigma^1$ generates a differential ideal.
  Subcase 2: if $rank(\mathbf A_1)=1$, then the system \eqref{n-di-sigma-1-condition-3} gives a relation between the $r_1$ and $r_2$, $r_2=h_2r_1$ and so affects the dimensions of the submodule in the next step, $\tilde \Sigma^2$, that is $dim (\tilde \Sigma^2)=n-1$ and $\tilde \Sigma^2:=Sp\{\tilde \omega^1, \omega^\alpha: \alpha=3,\dots,n\}$, where $\tilde\omega^1:=\omega^1+h_2\omega^2$, and where  $h_2$ is a known function.
   Subcase 3: if $rank(\mathbf A_1)=2$, then the solutions of the system \eqref{n-di-sigma-1-condition-3} are $r_1=0$ and $r_2=0$ which is a non-existence case.

\subsection{Case BNII1}\label{section-case-NII2}
Now we analyse the subcase 2 mentioned above, where $\mathbf \Phi$ is diagonalisable with distinct (real) eigenvalues with exactly two non-integrable eigen co-distributions and $rank(\mathbf A_1)=1$. Thus a differential ideal is not obtained at the first step, that is, $\tilde\Sigma^1:=Sp\{\phi^{cV}\wedge\phi^{cH}: c=1,\dots,n\}$ does not generate a differential ideal. The results are given in theorem \ref{case-2a2-n-results-2} at the end of the section followed by illustrative examples.

Solving this system \eqref{n-di-sigma-1-condition-3} with the condition that $r_1$ and $r_2$ are non-zero for a non-degenerate solution, with the assumption that $rank(\mathbf A_1)=1$, we get an equation relating $r_1$ and $r_2$, $r_2=h_2r_1$, and the conditions on the $\tau$'s are as follows,
\begin{equation}
\label{cond-domega-in-sigma1-1}\tau^{1\Gamma}_\alpha=0, \qquad \tau^{2\Gamma}_\alpha=0, \ \text{for all } \alpha=3,\dots,n
\end{equation} and in each of equation in the system \eqref{n-di-sigma-1-condition-3}, the coefficients of $r_1$ and $r_2$ must be both non-zero or both zero. Besides, since $rank(\mathbf A_1)=1$ by assumption, at least one of the ratios
\begin{equation}
\label{cond-domega-in-sigma1-2}-\frac{\tau^{1\Gamma}_2}{\tau^{2\Gamma}_1}, \ \frac{\tau^{1V}_{2\alpha}-\tau^{1V}_{\alpha 2}}{\tau^{2V}_{\alpha 1}},\  \frac{\tau^{1V}_{2\alpha}}{\tau^{2V}_{1\alpha}},\
\frac{\tau^{1H}_{2\alpha}-\tau^{1H}_{\alpha 2}}{\tau^{2H}_{\alpha 1}},\  \frac{\tau^{1H}_{2\alpha}}{\tau^{2H}_{1\alpha}}, \ \frac{\phi^{1V}(R(X_2^H,X_\alpha^H))}{\phi^{2V}(R(X_1^H,X_\alpha^H))}
\end{equation}
for all $\alpha=3,\dots,n$, must be well-defined and non-zero and when they are non-zero, they must be equal for all  those $\alpha$. Therefore $h_2$ equals the non-zero expressions.
So we assume that the conditions \eqref{cond-domega-in-sigma1-1} and \eqref{cond-domega-in-sigma1-2} hold, we have
$$\tilde\Sigma^2:=Sp\{\tilde{\omega}^1,\omega^\alpha: \alpha=3,\dots,n\},\quad \text{where } \tilde{\omega}^1=\omega^1+h_2\omega^2.$$
Now consider the conditions for $\langle\tilde\Sigma^2\rangle$ to be a differential ideal.
Let $\omega \in \tilde \Sigma^2$, then $\omega=\tilde r_1 \tilde\omega^1+r_\alpha\omega^\alpha$, $\alpha \text{ summed } 3,\dots,n$. Calculating the exterior derivative of $\omega$, we have
\begin{align*}
d\omega=&d\tilde r_1\wedge\tilde\omega^1+\tilde r_1d \tilde\omega^1+dr_\alpha\wedge\omega^\alpha+r_\alpha d\omega^\alpha\\
       =&(d\tilde r_1+\tilde r_1 \tilde\xi^1_1)\wedge\tilde\omega^1+(dr_\alpha+\tilde r_1\tilde\xi^1_\alpha+r_\alpha\xi^\alpha_\alpha)\wedge\omega^\alpha+\tilde r_1(dh_2+\tilde\xi^1_2-h_2\tilde\xi^1_1)\wedge\omega^2,
\end{align*}
for $\alpha=3,...,n$ and where
\[\tilde\xi^1_c:=\xi^1_c+h_2\xi^2_c,\] 
  for each $c=1,\dots,n$. Thus $d\omega\in\langle\tilde\Sigma^2\rangle$ for all $\omega \in \tilde \Sigma^2$ (so $\langle\tilde\Sigma^2\rangle$ is a differential ideal) if and only if,

\begin{equation*}\label{cond-sigma2-DI}
d\tilde{\omega}^1=\tilde{\xi}^1_1 \wedge \tilde{\omega}^1+\tilde{\xi}^1_\alpha \wedge \omega^\alpha, \quad \alpha=3,\dots,n.
\end{equation*}
This condition is equivalent to
\begin{align}
\notag &dh_2+\xi^1_2+h_2(\xi^2_2-\xi^1_1-h_2\xi^2_1)\equiv 0 \quad (\text{mod }\phi^{2V},\phi^{2H})\\
\label{NII-sigma2-condition}\Leftrightarrow \  & dh_2+\xi^1_2+h_2(\xi^2_2-\xi^1_1)\equiv 0 \quad (\text{mod }\phi^{2V},\phi^{2H}),
\end{align}
as $\xi^2_1=\tau^{2V}_{11}\phi^{2V}+\tau^{2H}_{11}\phi^{2H}\equiv 0$ (mod $\phi^{2V},\phi^{2H}$).

Now let us assume that the condition \eqref{NII-sigma2-condition} holds, this means $\tilde \Sigma^2$ is the final submodule.
The next step is to find the non-degenerate and closed forms in $\tilde \Sigma^2$ by solving the system of Pfaffian equations
\begin{align}
&d\tilde r_1+\tilde r_1\tilde \xi^1_1=0 \label{n-p-1-pfaff-1}\\
&dr_\alpha+\tilde r_1 \tilde\xi^1_\alpha+r_\alpha\xi^\alpha_\alpha=-P_\alpha\phi^{\alpha V}-Q_\alpha\phi^{\alpha H} \ \text{(no sum on $\alpha$)}\label{n-p-1-pfaff-b}
\end{align}
where $\alpha=3,...,n$ and $P_\alpha, Q_\alpha$ are arbitrary functions.

Following the EDS procedure, we extend $E$ to a new manifold $N$ with coordinates $(t,x^c,u^c,r_c,P_c,Q_c)$ and now the problem is to find the integrable distributions on $N$ with $\sigma_\alpha=0, \alpha=3,\dots,n$ and $\sigma_1=0$ where
\begin{align}
\sigma_1:=&d\tilde r_1+\tilde r_1 \tilde\xi^1_1  \label{n-p-1-EDS-sigma-1}\\
\sigma_\alpha:=&dr_\alpha+\tilde r_1\tilde\xi^1_\alpha+r_\alpha\xi^\alpha_\alpha+P_\alpha\phi^{\alpha V}+Q_\alpha\phi^{\alpha H} \label{n-p-1-EDS-sigma-b}
\end{align}

Continuing the EDS process, set $\pi_\alpha^P:=dP_\alpha$ and $\pi_\alpha^Q:=dQ_\alpha$, $\alpha=3,...,n$. Using this a co-frame on $N$ is $(dt, \phi^{dV}, \phi^{dH}, \sigma_d,\pi^P_d, \pi^Q_d)$ for $d=1,...,n$. So the next step is to calculate $d\sigma_1$ and $d\sigma_\alpha$ modulo $\{\sigma_1,\sigma_\alpha: \alpha=3,...,n\}$, as follows:

Taking the exterior derivative of \eqref{n-p-1-EDS-sigma-1} and \eqref{n-p-1-EDS-sigma-b} gives:
\begin{equation*}
d\sigma_1\equiv\tilde r_1d \tilde\xi^1_1 \ (\text{mod } \sigma_1)
\end{equation*}
and, for each $\alpha=3,...,n$,
\begin{align*}
d\sigma_\alpha=&d\tilde r_1\wedge \tilde\xi^1_\alpha+\tilde r_1d \tilde\xi^1_\alpha+dr_\alpha\wedge\xi^\alpha_\alpha+r_\alpha d\xi^\alpha_\alpha\\
          &+dP_\alpha\wedge\phi^{\alpha V}+P_\alpha d\phi^{\alpha V}+dQ_\alpha\wedge\phi^{\alpha H}+Q_\alpha d\phi^{\alpha H} \ (\text{no sum on }\alpha)\\
         \equiv &-\tilde r_1\tilde\xi^1_1\wedge \tilde\xi^1_\alpha+\tilde r_1d\tilde\xi^1_\alpha\\
         &+(-\tilde r_1\tilde\xi^1_\alpha-r_\alpha\xi^\alpha_\alpha-P_\alpha\phi^{\alpha V}-Q_\alpha\phi^{\alpha H})\wedge\xi^\alpha_\alpha+r_\alpha d\xi^\alpha_\alpha\\
         &+\pi^P_\alpha\wedge\phi^{\alpha V}+P_\alpha d\phi^{\alpha V}+\pi^Q_\alpha\wedge\phi^{\alpha H}+Q_\alpha d\phi^{\alpha H} \ (\text{mod } \sigma_1, \sigma_\alpha) \ (\text{no sum on }\alpha)
\end{align*}
Now we see which terms in $d\sigma_\alpha$ can be absorbed into $\pi^P_\alpha$ and $\pi^Q_\alpha$.
Note that in each $d \sigma_\alpha$, any term that can be written as $\beta \wedge \phi^{\alpha V}$ or $\beta \wedge \phi^{\alpha H}$
can be absorbed into terms $\pi^P_\alpha \wedge \phi^{\alpha V}$ and $\pi^Q_\alpha \wedge \phi^{\alpha H}$ respectively.
After this absorption these terms are denoted as $\tilde{\pi}^P_\alpha \wedge \phi^{\alpha V}$  and $\tilde{\pi}^Q_\alpha \wedge \phi^{\alpha H}$ and the remainder that can't be absorbed represents the \textit{`torsion'} of the
system.

Recall that each eigen co-distribution $Sp\{\phi^{\alpha V},\phi^{\alpha H}\}$, $\alpha=3,\dots,n$, is integrable,
so as given at \eqref{comp-in-cond-2}-\eqref{dxi-aa}
we have that
\begin{align*}
d\phi^{\alpha V} \equiv 0, \qquad 
d\phi^{\alpha H} \equiv 0, \qquad 
d\xi^\alpha_\alpha \equiv 0 \quad (\text{mod } \phi^{\alpha V}, \phi^{\alpha H}).
\end{align*}

So the terms $d\phi^{\alpha V}$, $d\phi^{\alpha H}$ and $d\xi^\alpha_\alpha$  can be absorbed.

Thus, we have
\begin{equation*}
\begin{aligned}
d\sigma_\alpha\equiv &\tilde\pi^P_\alpha\wedge\phi^{\alpha V}+\tilde \pi^Q_\alpha\wedge\phi^{\alpha H}\\
             &+\tilde r_1\big[(\xi^\alpha_\alpha-\tilde\xi^1_1)\wedge\tilde\xi^1_\alpha+d\tilde \xi^1_\alpha\big] \ (\text{mod } \sigma_1, \sigma_\alpha)
\end{aligned}
\end{equation*}
The torsion must be zero for nontrivial solutions and enforcing non-degeneracy results in the following conditions
\begin{equation}\label{cond-torsion-vanish-1}
d\tilde\xi^1_1=0 \ \Leftrightarrow \  d(\xi^1_1+h_2\xi^2_1)=0,
\end{equation}
and
\begin{equation}\label{cond-torsion-vanish-2}
\begin{aligned}
&(\xi^\alpha_\alpha-\tilde\xi^1_1)\wedge\tilde\xi^1_\alpha+d\tilde \xi^1_\alpha\equiv 0 \ (\text{mod }\phi^{\alpha V},\phi^{\alpha H})\\
\Leftrightarrow \ &(\xi^\alpha_\alpha-\xi^1_1-h_2\xi^2_1)\wedge(\xi^1_\alpha+h_2\xi^2_\alpha)+d(\xi^1_\alpha+h_2\xi^2_\alpha)\equiv 0 \ (\text{mod }\phi^{\alpha V},\phi^{\alpha H})
\end{aligned}
\end{equation}
for each $\alpha=3,...,n$.

If we assume that all conditions \eqref{cond-domega-in-sigma1-1}, \eqref{cond-domega-in-sigma1-2}, \eqref{NII-sigma2-condition}, \eqref{cond-torsion-vanish-1} and \eqref{cond-torsion-vanish-2} are satisfied, we have
\begin{align*}
 &d\sigma_1\equiv 0 \quad (\text{mod }\sigma)\\
&d\sigma_\alpha\equiv \tilde{\pi}^P_\alpha\wedge \phi^{\alpha V}+\tilde{\pi}^Q_\alpha \wedge \phi^{\alpha H} \quad (\text{mod }\sigma)
\end{align*}
We change the basis $\{\phi^{cV},\phi^{cH}\}$  to the basis $\{\gamma^{cV},\gamma^{cH}\}$ using:
\begin{align*}
\gamma^{1V/H}=\phi^{1V/H}+\phi^{2V/H}+...+\phi^{nV/H}\\
\gamma^{dV/H}=\phi^{1V/H}-\phi^{dV/H}, \quad d=2,...,n
\end{align*}
We then get the optimal tableau:
\begin{center}
\renewcommand{\arraystretch}{1.25}
$\tilde \Pi = $
\begin{tabular}{c|c c c c c c c c c c  }
           & $\gamma^{1V}$  & $\gamma^{1H}$  & ...& $\gamma^{pV}$ & $\gamma^{pH}$ &$\gamma^{(p+1)V}$ & $\gamma^{(p+1)H}$  &...& $\gamma^{nV}$ & $\gamma^{nH}$  \\ \hline
$\sigma_1$ &0&0& ...    &     0   &0    & 0   & 0&...& 0   & 0  \\
$\sigma_3$ &     $\tilde \pi^P_3$     &     $\tilde \pi^Q_3$     & ... &  0  & 0    & $-\tilde \pi^P_3$   &   $-\tilde \pi^Q_3$&...& $0$   & $0$  \\
\vdots     &    \vdots   &  \vdots     &  ...&  \vdots     & \vdots    &  \vdots &  \vdots&...&  \vdots &  \vdots \\
$\sigma_n$ &     $\tilde \pi^P_n$     &     $\tilde \pi^Q_n$     &      ...    &      $0$    & $0$    & $0$   & $0$ &...&  $-\tilde \pi^P_n$  &  $-\tilde \pi^Q_n$
\end{tabular}
\end{center}
This tableau gives Cartan characters: $s_1=n-2$, $s_2=n-2$, $s_i=0$ for $i=3,...,n$.

The final step is to check for involution. To do this, we let $t$ be the number of ways that $\tilde \pi^P_\alpha$ and $\tilde\pi^Q_\alpha$ can be altered such that \eqref{NII-dsigma-b} are unchanged. It can be seen that if we write:
\begin{align*}
\bar\pi^{P}_\alpha&=\tilde\pi^{P}_\alpha+f_\alpha^1\phi^{\alpha V}+f_\alpha^2\phi^{\alpha H},\\
\bar\pi^{Q}_\alpha&=\tilde\pi^{Q}_\alpha+f_\alpha^3\phi^{\alpha H}+f_\alpha^2\phi^{\alpha V},\\
\end{align*}
then \eqref{NII-dsigma-b} would be unchanged if we replace $\tilde \pi^{P/Q}_\alpha$ by $\bar \pi^{P/Q}_\alpha$.
Thus for each $\alpha =3,...,n$ we have three degrees of freedom in modifying $\tilde\pi^{P}_\alpha$ and $\tilde\pi^{Q}_\alpha$, giving $3(n-2)$ degrees of freedom for all $\tilde \pi^{P/Q}_\alpha$.
Therefore in this case, $t=3(n-2)$, which equal to $s_1+2s_2$ as required for involution.
So the solution depends on $n-2$ functions of two variables in this case.

In summary, the result of this case is given in the following theorem.

\begin{thm}\label{case-2a2-n-results-2}
Assume that $\mathbf \Phi$ is diagonalisable with distinct (real) eigenvalues and with exactly $2$ non-integrable eigen co-distributions.
Suppose that eigen co-distributions are ordered such that $Sp\{\phi^{1V},\phi^{1H}\}$ and $Sp\{\phi^{2V},\phi^{2H}\}$ are non-integrable.
Suppose further that $\langle\tilde\Sigma^1\rangle$ where $\tilde\Sigma^1:=Sp\{\phi^{cV}\wedge\phi^{cH}:c=1,\dots,n\}$ is not a differential ideal.
Then the necessary and sufficient for the existence of a solution for the associated inverse problem are that
the conditions \eqref{cond-domega-in-sigma1-1}, \eqref{cond-domega-in-sigma1-2}, \eqref{NII-sigma2-condition}, \eqref{cond-torsion-vanish-1} and \eqref{cond-torsion-vanish-2} hold.
Furthermore, the solution (if it exists) depends on $n-2$ arbitrary functions of $2$ variables each.
\end{thm}
\begin{xmpl}
This is a non-existence example of the case B3II1 as the condition for $\langle\tilde \Sigma^2\rangle$ to be a differential ideal \eqref{NII-sigma2-condition} fails. \\
We consider the following system
\begin{equation}\label{n-NII-xmpl-1}
\ddot{x}=x\dot{y}, \
\ddot{y} =\dot{x},\
\ddot{z}=0,
\end{equation}
on an appropriate domain. Denoting $\dot x,\dot y,\dot z$  by $u,v,w$, we find that
$\mathbf \Phi$ is diagonalisable with distinct eigenvalues and corresponding re-scaled eigenvectors $X_a$ as follows,
\begin{eqnarray*}
\lambda_1=-\frac{x}{4} &\quad\text{and}\quad& X_1=(\frac{u}{\sqrt[4]{v^3}},\sqrt[4]{v},0), \\
\lambda_2=-\frac{4v+x}{4} &\quad\text{and}\quad& X_2=(\frac{1}{\sqrt[4]{v}},0,0),  \\
\lambda_3=0 &\quad\text{and}\quad& X_3=(0,0,1).
\end{eqnarray*}
The structure functions $\tau$'s are zero except for
\[
\tau^{1\Gamma}_{2}=-\frac{\sqrt{v}}{4v},\ \tau^{2\Gamma}_{1}=-\frac{3u^2}{4v\sqrt{v}},\ \tau^{1H}_{11}=\frac{u}{8v\sqrt[4]{v^3}},\ \tau^{1V}_{11}=\frac{1}{2\sqrt[4]{v^3}},\ \tau^{1H}_{21}=\frac{1}{8v\sqrt[4]{v}}\]
\[\tau^{2H}_{11}=\frac{2xv^2-u^2}{2v^2\sqrt[4]{v}},\ \tau^{2V}_{11}=\frac{-u}{v\sqrt[4]{v}},\ \tau^{2H}_{12}=\frac{-u}{8v\sqrt[4]{v^3}},\ \tau^{2V}_{12}=-\frac{1}{2\sqrt[4]{v^3}}\]
\[\tau^{2V}_{22}=-\frac{1}{8v\sqrt[4]{v}},\ \phi^{2V}(R(X_1^H,X_2^H))=-\sqrt[4]{v}\]

These results show that the two eigen co-distributions $Sp\{\phi^{1V}, \phi^{1H}\}$ and $Sp\{\phi^{2V}, \phi^{2H}\}$ are non-integrable and the third one is integrable by proposition \ref{comp-in-cond} and $\langle\Sigma^1\rangle$ is not a differential ideal by corollary \ref{firststep-condition} and that the conditions \eqref{cond-domega-in-sigma1-1} and \eqref{cond-domega-in-sigma1-2} hold with $$h_2=-\frac{\tau^{1\Gamma}_2}{\tau^{2\Gamma}_1}=-\frac{v}{3u^2}.$$ Further examination is whether or not the condition \eqref{NII-sigma2-condition} holds.
Calculations gives
$$dh_2=d(-\frac{v}{3u^2})\equiv \frac{2xv^2-u^2}{3u^3}dt+\frac{x(4v^2-u^3)}{6u^5\sqrt[4]{v^3}}\phi^{1H} \quad (\text{mod }\phi^{2V},\phi^{2H}),$$
\begin{align*}
\xi^2_2=&A^{2V}_{21}\phi^{1V}+A^{2H}_{21}\phi^{1H}+A^{2V}_{23}\phi^{3V}+A^{2H}_{23}\phi^{3V}\\
        =&\frac{1}{\sqrt[4]{v^3}}\phi^{1V}+\frac{u}{4v\sqrt[4]{v^3}}\phi^{1H},
\end{align*}
\begin{align*}
\xi^1_1=&A^{1V}_{12}\phi^{2V}+A^{1H}_{12}\phi^{2H}+A^{1V}_{13}\phi^{3V}+A^{1H}_{13}\phi^{3V}\\
        =&-\frac{1}{\sqrt[4]{v}}\phi^{2H},\\
\xi^1_2       =&0.
\end{align*}
Thus, the condition \eqref{NII-sigma2-condition} does not hold and so $\langle\tilde\Sigma^2\rangle$ is not a differential ideal.  Therefore, the corresponding inverse problem of this system of second-order ordinary differential equations \eqref{n-NII-xmpl-1} has no regular solutions.
\end{xmpl}
\begin{xmpl}
We consider another example of the subcase BNII1 analysed above. This B3II1 system was introduced by us in \cite{DP16},
\begin{equation*}
\ddot{x}=zt, \
\ddot{y} = 0,\
\ddot{z}=x,
\end{equation*}
on an appropriate domain.
Denoting the derivatives by $u,v,w$, we find that
$\mathbf \Phi$ is diagonalisable with distinct eigenvalues and corresponding eigenvectors $X_a$ as follows,
\begin{eqnarray*}
\lambda_1=\sqrt{t} &\quad\text{and}\quad& X_1=(-\sqrt{t},0,1), \\
\lambda_2=-\sqrt{t} &\quad\text{and}\quad& X_2=(\sqrt{t},0,1),  \\
\lambda_3=0 &\quad\text{and}\quad& X_3=(0,1,0).
\end{eqnarray*}
The structure functions $\tau$'s are zero except for

\[
\tau^{1\Gamma}_{1}=\tau^{2\Gamma}_{2}=-\tau^{1\Gamma}_{2}=-\tau^{2\Gamma}_{1}=\frac{1}{4t}.\]

These results show that the eigen co-distributions $Sp\{\phi^{1V}, \phi^{1H}\}$ and $Sp\{\phi^{2V}, \phi^{2H}\}$ are non-integrable and the third one is integrable by proposition \ref{comp-in-cond}. Also $\langle\tilde \Sigma^1\rangle$ is not a differential ideal by corollary \ref{firststep-condition}, and that the conditions \eqref{cond-domega-in-sigma1-1} and \eqref{cond-domega-in-sigma1-2} hold with $h_2=-1$. Further examination gives
\[d\tilde{\omega}^1=-\frac{1}{2t} dt \wedge \tilde{\omega}^1, \quad \tilde{\omega}^1=\omega^1-\omega^2,\]
that is, the condition \eqref{cond-sigma2-DI} holds with $\tilde{\xi}^1_1=-\frac{1}{2t} dt$ and $\tilde{\xi}^1_3=0$ and so $\tilde \Sigma^2:=Sp\{\tilde{\omega}^1, \omega^3\}$ generates a differential ideal. The remaining conditions \eqref{cond-torsion-vanish-1} and \eqref{cond-torsion-vanish-2} also hold for solution as $\tilde{\xi}^1_1=-\frac{1}{2t} dt$ and $\tilde{\xi}^1_3=0$. Therefore this system is variational and  the solution depends on one arbitrary function of two variables.

To determine the explicit expression of the Cartan two-form for this example, 
we examine the Pfaffian equations \eqref{n-p-1-pfaff-1} and \eqref{n-p-1-pfaff-b}. Explicitly, in this example, they are
 \begin{align*}
 &d\tilde r_1+\tilde r_1\tilde\xi^1_1=0,\\
 &dr_3+P_3\phi^{3V}+Q_3\phi^{3H}=0
 \end{align*}
 We then find that $\tilde r_1=G\sqrt{t}$ where $G$ is a constant and $r_3=r_3(u_3^1,u_3^2)$ is an arbitrary function of two variables $u^1_3=y-vt$ and $u^2_3=v$. Thus the Cartan 2-form finally is \[\omega=G\sqrt{t}(\omega^1-\omega^2)+r_3(u_3^1,u_3^2)\omega^3. \]
\end{xmpl}

In the next section, we will present the results for subcase 1 of case BNII in which the rank of the system \eqref{n-di-sigma-1-condition-3} is zero, that is,  $\langle\tilde \Sigma^1\rangle$ is differential ideal.

\subsection{Case BNII0}
This is  the case where $\mathbf \Phi$ is diagonalisable with distinct eigenvalues, two non-integrable
eigen co-distributions, $n-2$ integrable eigen co-distributions and $rank(A_1)=0$, that is $\tilde\Sigma^1:=Sp\{\omega^a:a=1,\dots,n\}$ generates a
differential ideal.  In the case $n=2$, this corresponds to the most difficult case of Douglas, case IIa3, and  not entirely complete in his paper.
Again, we assume that $Sp\{\phi^{1V},\phi^{1H}\}$ and $Sp\{\phi^{2V},\phi^{2H}\}$ are non-integrable
and $Sp\{\phi^{\alpha V},\phi^{\alpha H}: \alpha=3, \dots, n\}$ are integrable.

A full analysis can been found in~\cite{Do16}, we restrict ourselves to stating an abbreviated version of the main result and an example in $n=3$. The `certain conditions' referred to in the theorem below correspond for this case to the conditions in theorem
\ref{case-2a2-n-results-2} for case BNII1.
\begin{thm}\label{last thm}
 Assume that $\mathbf \Phi$ is diagonalisable with distinct eigenvalues having 2 non-integrable eigen co-distributions, $n-2$ integrable co-distributions and $rank(A_1)=0$.
 The existence of solutions to the inverse problem depends on whether or not certain conditions (see~\cite{Do16}) are satisfied. The solution (if it exists) depends on $n-2$ functions of two variables.
\end{thm}

\begin{xmpl}\label{ex-3AIII20-1}
This example is a straightforward modification of a case B2II0 example where the added equation produces an integrable eigen co-distribution.
Consider the system
\begin{equation*}
\ddot{x}=x\dot{z}, \
\ddot{y} = x, \
\ddot{z}=x
\end{equation*}
on an appropriate domain. Again denoting the derivatives by $u,v,w$, we find
\[
\mathbf \Phi = \left( \begin{array}{ccc} -w & 0 & \frac{u}{2} \\
-1 & 0 & 0 \\
-1&0&0 \end{array}\right)
\]
 is diagonalisable with distinct eigenvalues and corresponding eigenvectors $X_a$ chosen so that $\hnabla_\Gamma X_a^V=0$:
\begin{eqnarray*}
\lambda_1=\sqrt{-2u+w^2}-w &\quad\text{and}\quad& X_1=(-\sqrt{-2u+w^2}+w,2,2), \\
\lambda_2=-\sqrt{-2u+w^2}-w &\quad\text{and}\quad& X_2=(\sqrt{-2u+w^2}+w,2,2),  \\
\lambda_3=0 &\quad\text{and}\quad& X_3=(0,1,0).
\end{eqnarray*}

The non-zero functions $\tau^{aV}_{bc}$ and $\tau^{aH}_{bc}$ are

\[
\tau^{1V}_{11}=-\tau^{2V}_{11}=\frac{\sqrt{-2u+w^2}-w}{2(2u-w^2)}, \quad \tau^{1H}_{11} =-\tau^{2H}_{11}=\frac{x}{2(2u-w^2)},\]
\[\tau^{1V}_{12}=-\tau^{2V}_{12}=\frac{3\sqrt{-2u+w^2}+w}{2(2u-w^2)}, \quad \tau^{1H}_{12}=-\tau^{2H}_{12}=\frac{-x}{2(2u-w^2)},\]
\[\tau^{1V}_{21}=-\tau^{2V}_{21}=\frac{3\sqrt{-2u+w^2}-w}{2(2u-w^2)}, \quad \tau^{1H}_{21}=-\tau^{2H}_{21}=\frac{x}{2(2u-w^2)},\]
\[\tau^{1V}_{22}=-\tau^{2V}_{22}=\frac{\sqrt{-2u+w^2}+w}{2(2u-w^2)}, \quad \tau^{1H}_{22}=-\tau^{2H}_{22}=\frac{-x}{2(2u-w^2)}.\]

These results show that the eigen co-distributions $Sp\{\phi^{1V}, \phi^{1H}\}$ and $Sp\{\phi^{2V}, \phi^{2H}\}$ are
non-integrable and the third one is integrable and $\langle\tilde \Sigma^1\rangle$ is differential ideal.
 Furthermore, the existence conditions of theorem~\ref{last thm} are also satisfied. So, the solution depends on one arbitrary function of two variables.

\end{xmpl}

\section*{Acknowledgements}
We thank Willy Sarlet for useful discussions and his continuing interest.


\begin{thebibliography}{99}


\bibitem{Ald03}
J.E.\ Aldridge. {\em Aspects of the Inverse Problem in the Calculus of Variations} Ph.D. Thesis, La Trobe University,
Australia (2003).

\bibitem{APST06}
J.E.\ Aldridge, G.E.\ Prince, W. Sarlet and G. Thompson. An EDS approach to the inverse problem in the calculus of variations,
{\em  J.\ Math.\ Phys. \/} {\bf 47} (2006) 103508 (22 pages).


\bibitem{AT92}
I.\ Anderson and G.\ Thompson. The inverse problem of the calculus
of variations for ordinary differential equations, {\em  Memoirs
Amer.\ Math.\ Soc.\/} {\bf 98} No. 473 (1992).


\bibitem{Bryant91}
R.L.\ Bryant, S.S.\ Chern, R.B.\ Gardner, H.L.\ Goldschmidt and
P.A.\ Griffiths. {\em  Exterior Differential Systems} (Springer,
Berlin)  (1991).

\bibitem{CMS96}
M. \ Crampin, E. \ Mart\'{\i}nez and W. \ Sarlet.
Linear connections for systems of second--order ordinary
differential equations. {\em Ann. Inst. H. Poincar\'e Phys. Th\'eor.\/}
{\bf 65} (1996), 223--249.

\bibitem{CPST}
M.\ Crampin, G.E.\ Prince, W.\ Sarlet and G.\ Thompson. The
inverse problem of the calculus of variations: separable systems,
{\em Acta Appl.\ Math.\/} {\bf 57} (1999) 239--254.

\bibitem{CPT84}
M.\ Crampin, G.E.\ Prince and G.\ Thompson. A geometric version of
the Helmholtz conditions in time dependent Lagrangian dynamics,
 {\em  J.\ Phys.\ A:\ Math.\ Gen. \/} {\bf 17} (1984) 1437--1447.

\bibitem{CSMBP94}
M. Crampin, W. Sarlet, E. Mart\'{\i}nez, G. B. Byrnes and G. E. Prince.
Toward a geometrical understanding of Douglas's solution of the
inverse problem in the calculus of variations.
{\em  Inverse Problems\/} {\bf 10} (1994) 245--260.

\bibitem{Do16}
T. Do. {\em The Inverse Problem in the Calculus of Variations
via Exterior Differential Systems.} Ph.D. Thesis, La Trobe University, Australia (2016)

\bibitem{DP16}
T. Do and G.E.\ Prince. New progress in the inverse problem in the calculus of variations.
{\em Diff. Geom. Appl.\/} {\bf 45}(2016)148--179.

\bibitem{D41}
J.\ Douglas. Solution of the inverse problem of the calculus of
variations, {\em  Trans.\ Am.\ Math.\ Soc.\/} {\bf 50} (1941)
71--128.

\bibitem{HH01}
H. Helmholtz. \"{U}ber der physikalische Bedeutung des Princips der kleinsten
Wirkung, {\it J. Reine Angew. Math.} {\bf 100} (1887) 137--166.

\bibitem{H82}
M. Henneaux. On the inverse problem of the calculus of variations,
{\em J. Phys. A: Math. Gen.\/ }{\bf 15} (1982) L93--L96.

\bibitem{HS88}
M. Henneaux and L. C. Shepley. Lagrangians for spherically symmetric potentials,
{\em J. Math. Phys.\/ }{\bf 23}, (1988) 2101--2107.


\bibitem{Hirsch02}
A. Hirsch.  Die Existenzbedingungen des verallgemeinterten kinetischen Potentialen, {\it Math. Ann.} {\bf
50} (1898) 429--441.

\bibitem{JP01}
M.\ Jerie and G.E.\ Prince. Jacobi fields and linear connections
for arbitrary second order ODE's, {\em J.\ Geom.\ Phys.\/} {\bf
43} (2002) 351--370.

\bibitem{KP08}
O. Krupkov\'{a} and G.E.\ Prince. Second order ordinary differential equation in jet bundles and the inverse problem of the calculus of variation in: {\it Handbook of Global Analysis}, edited by D. Krupka and D. Saunders, Elsevier 2008.

\bibitem{MaPa94}
E. Massa and E. Pagani. Jet bundle geometry, dynamical connections, and the inverse problem of Lagrangian mechanics, {\em Ann. Inst. Henri Poincar\'{e}, Phys. Theor.} {\bf 61} (1994) 17--62.

\bibitem{MFLMR90}
G. Morandi, C. Ferrario, G. Lo Vecchio, G. Marmo and C. Rubano.
The inverse problem in the calculus of variations and the geometry of
the tangent bundle.
{\em  Phys.\ Rep.\/} {\bf 188} (1990) 147--284.

\bibitem{S82}
W.\ Sarlet. The Helmholtz conditions revisited. A new approach to
the inverse problem of Lagrangian dynamics, {\em  J.\ Phys.\ A:
Math.\ Gen.\/} {\bf 15} (1982) 1503--1517.

\bibitem{SaCraMa}
W.\ Sarlet, M.\ Crampin and E.\ Mart\'{\i}nez. The integrability
conditions in the inverse problem of the calculus of variations
for second-order ordinary differential equations, {\it Acta Appl.\
Math.\/} {\bf 54} (1998) 233--273.

\bibitem{STP}
W. Sarlet, G. Thompson and G.E. Prince. The inverse problem of
the calculus of variations: the use of geometrical calculus in
Douglas's analysis, {\em Trans.\ Amer.\ Math.\ Soc.\/} {\bf 354}
(2002) 2897--2919.


\bibitem{Son}
N. Ya. Sonin. On the definition of maximal and minimal properties,
{\em Warsaw Univ. Izvestiya} {\bf 1--2} (1886) {1--68 (in Russian)}
%
\end{thebibliography}
\end{document}